\documentclass[
]{amsart}
\usepackage{algorithm, algorithmic}
\usepackage{amsmath}\usepackage{amsthm}\usepackage{amssymb}\usepackage{mathrsfs}\usepackage{amscd}\usepackage{graphicx}\usepackage{subfigure}\usepackage{amsfonts}\usepackage{amsxtra}\usepackage{color}
\usepackage{amscd}
\usepackage{hyperref}

\DeclareMathOperator{\wce}{wce}

\DeclareMathOperator{\vol}{vol}
\DeclareMathOperator{\G}{\mathcal{G}}

\DeclareMathOperator{\trace}{trace}

\DeclareMathOperator*{\spann}{span}\DeclareMathOperator{\dist}{dist}\DeclareMathOperator{\supp}{supp}\DeclareMathOperator{\tr}{\trace}

\theoremstyle{definition}
\newtheorem{definition}{Definition}[section]
\newtheorem{remark}[definition]{Remark}
\theoremstyle{plain}\newtheorem{theorem}[definition]{Theorem}\newtheorem{lemma}[definition]{Lemma}\newtheorem{proposition}[definition]{Proposition}

\newcommand{\R}{\mathbb{R}}\newcommand{\C}{\mathbb{C}}

\begin{document}

\title[asymptotically optimal covering radius]{Points on manifolds with asymptotically optimal covering radius
}

\author[A.~Breger]{Anna Breger}
\address[A.~Breger]{University of Vienna,
Department of Mathematics,
Oskar-Morgenstern-Platz 1
A-1090 Vienna
}
\email{anna.breger@univie.ac.at}

\author[M.~Ehler]{Martin Ehler}
\address[M.~Ehler]{University of Vienna,
Department of Mathematics,
Oskar-Morgenstern-Platz 1
A-1090 Vienna
}
\email[Corresponding author]{martin.ehler@univie.ac.at}

\author[M.~Gr\"af]{Manuel Gr\"af}
\address[M.~Gr\"af]{University of Vienna,
Department of Mathematics,
Oskar-Morgenstern-Platz 1
A-1090 Vienna
}
\email{manuel.graef@univie.ac.at}

\begin{abstract}
Given a finite set of points on the Euclidean sphere, the worst case quadrature error in Sobolev spaces has recently been shown to provide upper bounds on the covering radius of the point set. Moreover, quasi-Monte Carlo integration points on the sphere achieve the asymptotically optimal covering radius. Here, we extend these results to points on compact smooth Riemannian manifolds and provide numerical experiments illustrating our findings for the Grassmannian manifold. 
\end{abstract}


\maketitle


\section{Introduction}
\label{sec:1}
Many discretization schemes in numerical analysis are based on finite samples
that cover the underlying space sufficiently well, i.e., the sampling points
have small covering radius. One way of measuring the covering's efficiency is by
its cardinality in comparison to its covering radius.

Quasi-Monte Carlo integration points have been investigated in
\cite{Brandolini:2014oz} for compact smooth Riemannian manifolds. In the special
case of the Euclidean sphere, it is shown recently in \cite{Brauchart:2015ec}
that the worst case error of integration bounds the covering radius and that
thereby quasi-Monte Carlo integration points provide asymptotically optimal
covering radii.

Here, we extend these results from the sphere to compact smooth Riemannian manifolds. 
These theoretical results are derived by combining the ideas in \cite{Brandolini:2014oz} with those in \cite{Brauchart:2015ec}. 

In the second part of the present note, we numerically construct a sequence of quasi-Monte Carlo integration points for the Grassmannian manifold and illustrate numerically that their covering radii indeed behave in accordance to the theoretical findings, hence, asymptotically optimal. 

Our quasi-Monte Carlo integration points are cubatures (in fact designs) in Grassmannians that have been studied in \cite{Bachoc:2005aa,Bachoc:2004fk,Bachoc:2002aa,Bachoc:2010aa} from a theoretical point of view, see \cite{Ehler:2014zl} for the construction through numerical minimization. For related results on cubatures in more classical settings, see \cite{Dung:2015fe,Hinrichs:2015bf,Hinrichs:2014oj,Krieg:2016sf,Nowak:2010rr,Sloan:2004qd,Ullrich:2014oj} and, for further related results, we refer to \cite{Brauchart:2015qp,Choirat:2013oq,Damelin:2003lh,Filbir:2010aa,Filbir:2011fk,Hellekalek:2016ax,Niederreiter:2003hb} and references therein.

The outline is as follows. In Section \ref{sec:2}, we briefly discuss the concept of asymptotically optimal covering radii. Low-cardinality cubature points are introduced in Section \ref{sec:3}, where we also state our main theoretical result. In Section \ref{sec:wce} we state the bound on the covering radius by the worst case error of integration. Section \ref{sec:4} is dedicated to the proof of this bound. In Section \ref{sec:nums}, we illustrate our theoretical findings by numerical experiments for the Grassmannian manifold $\mathcal{G}_{2,4}$.

\section{Optimal asymptotics of the covering radius}\label{sec:2}
Let $\mathcal{M}$ be a compact smooth $d$-dimensional Riemannian manifold. We
denote its normalized Riemannian measure by $\mu$ and its Riemannian distance by
$\dist$. Given any finite collection of points
$\{x_j\}_{j=1}^n\subset \mathcal{M}$, the \emph{covering radius} $\rho$ is
\begin{equation*}
\rho:=\rho(\{x_j\}_{j=1}^n):=\sup_{x\in \mathcal{M}} \min_{1\leq j\leq n} \dist(x,x_j).
\end{equation*}
For $
B_r(x):=\{y\in\mathcal{M}:\dist(x,y)\leq r\} $ 
denoting the ball of radius $r$ centered at $x\in\mathcal{M}$, the union
$\bigcup_{j=1}^n B_\rho(x_j)$ covers $\mathcal{M}$
completely.\renewcommand*{\thefootnote}{\fnsymbol{footnote}}\setcounter{footnote}{1}
By compactness of $\mathcal M$ we have \footnote{ We use the notation
  $\lesssim$, meaning the left-hand side is less or equal to the right-hand side
  up to a positive constant factor. The symbol $\gtrsim$ is used analogously,
  and $\asymp$ means both hold, $\lesssim$ and $\gtrsim$. If not explicitly
  stated, the dependence or independence of the constants shall be clear from
  the context.}
\begin{equation}\label{eq:inequality balls}
  \mu(B_r(x)) \asymp r^d ,\quad x \in \mathcal M,\quad  0<r\leq 1,
\end{equation}
where the constants do not depend on $x$ or $r$. Hence, the line of inequalities
$\mu(\mathcal{M})\leq \sum_{j=1}^n \mu(B_\rho(x_j))\lesssim n\rho^d$ leads to
the lower bound
\begin{equation}\label{eq:low bound rho ho}
n^{-\frac{1}{d}}\lesssim \rho.
\end{equation} 

\begin{definition}
Given a sequence of $n_i$ points $\{x^i_j\}_{j=1}^{n_i}\subset\mathcal{M}$, $i=1,2,\ldots$, with $n_i\rightarrow\infty$, we say that the corresponding sequence of covering radii $\rho_i$ is \emph{asymptotically optimal} if the lower bound \eqref{eq:low bound rho ho} is matched, i.e., if
$
\rho_{i} \asymp n_i^{-\frac{1}{d}}.
$
\end{definition}

According to \cite{Reznikov:2015zr}, the expectation of the covering radius $\rho$ of $n$ random points $\{x_j\}_{j=1}^n$, independently identically distributed according to $\mu$, satisfies
\begin{equation}\label{eq:random asympt}
\mathbb{E}\rho \asymp n^{-\frac{1}{d}} \log(n)^{\frac{1}{d}}.
\end{equation}
Hence, there is an additional logarithmic factor, so that random points do not provide optimal covering radii. 

In this brief note, we shall verify that the recently introduced concept of quasi-Monte Carlo systems, cf.~\cite{Brandolini:2014oz,Brauchart:fk}, lead to point sets with asymptotically optimal covering radii. This generalizes results for the sphere in \cite{Brauchart:2015ec} to compact smooth Riemannian manifolds. 

\section{Optimal coverings from low-cardinality cubatures}\label{sec:3}

Let $\{\varphi_\ell\}_{\ell=0}^\infty$ be the collection of orthonormal
eigenfunctions of the Laplace-Beltrami operator $\Delta$ on $\mathcal{M}$ with
eigenvalues $\{-\lambda_\ell\}_{\ell=0}^\infty$ arranged by
$0=\lambda_0 \leq \lambda_1\leq \ldots$. We denote by $L_p(\mathcal{M})$,
$1\leq p<\infty$, the Banach space of complex-valued $\mu$-measurable functions
on $\mathcal{M}$, whose $p$-th power of the absolute value is integrable (with
the standard modifications when $p=\infty$).

The space of \emph{diffusion polynomials} of bandwidth $t \ge 0$ is
\begin{equation}\label{eq:poly diff}
\Pi_t:=\spann\{\varphi_\ell:\lambda_\ell\leq  t^2 \}.
\end{equation}
For $\{x_j\}_{j=1}^n\subset\mathcal{M}$ and weights
$\{\omega_j\}_{j=1}^n\subset\R$, we say that $\{(x_j,\omega_j)\}_{j=1}^n$ is a
\emph{cubature for $\Pi_t$} if 
  \begin{equation}\label{eq:def cub 00}
    \int_{\mathcal{M}} f(x)d\mu(x) = \sum_{j=1}^n \omega_j f(x_j),\quad \text{for all } f\in  \Pi_t.
  \end{equation}
The number $t$ refers to the \emph{strength} of the cubature.

Weyl's estimates on the spectrum of an elliptic operator yield 
\begin{equation*}
\dim(\Pi_t)\asymp t^d,
\end{equation*}
cf.~\cite[Theorem 17.5.3]{Hormander:1983gf}. Therefore, any sequence of
cubatures $\{(x^i_j,\omega^i_j)\}_{j=1}^{n_i}$ of strength $t_i$ must obey
$ n_i\gtrsim t_i^d $. 
\begin{definition}
We call a sequence of cubatures $\{(x^i_j,\omega^i_j)\}_{j=1}^{n_i}$ for $\Pi_{t_i}$ satisfying 
 \begin{equation}\label{eq:few cubs2}
n_{i}\asymp t_i^d
\end{equation}
with $n_i\rightarrow\infty$ a \emph{low-cardinality cubature sequence}. 
\end{definition}
The above definition makes sense since there do exist sequences of cubatures $\{(x^i_j,\omega^i_j)\}_{j=1}^{n_i}$ of strength $t_i$ with positive weights and $n_i\rightarrow \infty$ satisfying \eqref{eq:few cubs2}, cf.~\cite{Harpe:2005fk}. We now state our main theoretical result.
\begin{theorem}\label{th:basis}
  The covering radius of any low-cardinality cubature sequence with positive weights is asymptotically
  optimal.
\end{theorem}
The remaining part of the present paper is dedicated to prove Theorem \ref{th:basis} and to numerically illustrate our findings on the Grassmannian manifold. We conclude this section with a remark concerning the weights.
\begin{remark}
Cubatures of strength $t$, whose weights are all the same, $\omega_j=\frac{1}{n}$, for $j=1,\ldots,n$, are also called $t$-designs.  
If $\mathcal{M}=\mathbb{S}^{d}\subset\R^{d+1}$ is the unit Euclidean sphere, then there are $t$-designs satisfying \eqref{eq:few cubs2}, cf.~\cite{Bondarenko:2011kx}. By identifying $x\in  \mathbb{S}^{d}$ with $-x$, the analogous statement holds for the projective space. 
For general $\mathcal{M}$, however, we only know that $t$-designs exist, cf.~\cite{Seymour:1984bh}, but it is still an open problem whether or not the asymptotics \eqref{eq:few cubs2} can be achieved.
\end{remark}

\section{Worst case error of integration}\label{sec:wce}
To prove Theorem \ref{th:basis}, we follow the approach for the sphere in
\cite{Brauchart:2015ec}. We shall first introduce the worst case error of
integration and shall check that it provides an upper bound on the covering
radius. Next, we shall consider the concept of quasi-Monte Carlo points, i.e.,
points whose worst case error of integration decays sufficiently fast, so that
the covering radius is asymptotically optimal. Finally, we shall recapitulate from \cite{Brandolini:2014oz} 
that low-cardinality cubature points with positive weights are indeed
quasi-Monte Carlo points.

The \emph{worst case error} of integration of points
$\{x_j\}_{j=1}^{n}\subset \mathcal{M}$ and weights
$\{\omega_j\}_{j=1}^{n} \subset \R$ with respect to some Banach space
$\mathcal{H}$ of complex-valued functions on $\mathcal{M}$ is
\begin{equation}\label{eq:wce ee}
\wce(\{(x_j,\omega_j)\}_{j=1}^{n},\mathcal{H}) : = \sup_{\substack{f\in \mathcal{H}\\ \|f\|_{\mathcal{H}}\leq 1}}\Big| \int_{\mathcal{M}} f(x)d\mu(x)- \sum_{j=1}^{n} \omega_j f(x_j) \Big|.
\end{equation}
Although suppressed by our notation, \eqref{eq:wce ee} depends on the particular norm $\|\cdot\|_\mathcal{H}$, which shall always be clear from the context in the present manuscript. For most parts, we take 
$\mathcal{H}$ to be a Sobolev space, which we define next. The
Fourier transform of $f\in L_p(\mathcal{M})$ with $1\leq p\leq \infty$ is
\begin{equation*}
\hat{f}(\ell):=\int_{\mathcal{M}} f(x)\overline{\varphi_\ell(x)} d\mu(x),\qquad \ell=0,1,2,\ldots,
\end{equation*}
and extends to distributions on $\mathcal{M}$. The Sobolev space $W^s_p(\mathcal{M})$, for $1\leq p\leq \infty$ and $s>0$, is the set of all distributions on $\mathcal{M}$ with $(I+\Delta)^{s/2}f\in L_p(\mathcal{M})$, i.e., with 
\begin{equation}\label{eq:def norm sob}
\| f\|_{W^s_p}  := \|(I+\Delta)^{s/2} f \|_{L_{p}} = \|\sum_{\ell=0}^\infty  (1+\lambda_\ell)^{s/2} \hat{f}(\ell) \varphi_\ell \|_{L_p}<\infty.
\end{equation}
Note that $W^s_p(\mathcal{M})$ is contained in the space of continuous functions on $\mathcal{M}$ provided that $s>d/p$, cf.~\cite{Brandolini:2014oz}. We shall stick to this range throughout the present paper. 

It turns out that the covering radius is bounded by the worst case error. The following result has been derived in \cite{Brauchart:2015ec} for $\mathcal{M}$ being the Euclidean sphere and constant weights $\omega_j^i=1/n_i$. 
\begin{proposition}\label{prop:11}
  Let $1 \le p \le \infty$ and $s>d/p$ be given. Then for any points
  $\{x_{j}\}_{j=1}^{n} \subset \mathcal M$ and weights
  $\{\omega_j\}_{j=1}^{n} \subset \R$ with covering radius $\rho$ it holds
\begin{equation}\label{eq:b r 1}
\rho \lesssim \big[\wce(\{(x_j,\omega_j)\}_{j=1}^{n},W^s_p(\mathcal{M}))\big]^{1/(s+d/q)},
\end{equation}
where $1/p+1/q=1$. The constants may only depend on $\mathcal M$, $s$, and $p$.
\end{proposition}
We shall postpone the proof of Proposition \ref{prop:11} to Section \ref{sec:4}. At this point we turn to sequences of points whose worst case error of integration satisfies decay conditions, which connects to the covering radius via the bound \eqref{eq:b r 1}. The following definition is due to \cite{Brandolini:2014oz,Brauchart:fk}.
\begin{definition}
  Given $1\le p\le \infty$ and $s>d/p$, a sequence
  $\{(x^i_j,\omega^i_j)\}_{j=1}^{n_i}$, $i=1,2,\ldots$, of points
  $\{x_j^i\}_{j=1}^{n_i}\subset \mathcal{M}$ and weights
  $\{\omega_j^i\}_{j=1}^{n_i}\subset\R$ with $n_{i}\rightarrow \infty$ is called
  a \emph{quasi-Monte Carlo (qMC) system} for $W^s_p(\mathcal{M})$ if
\begin{equation}\label{eq:wce}
\wce(\{(x^i_j,\omega^i_j)\}_{j=1}^{n_i},W^s_p(\mathcal{M}))\lesssim n_i^{-\frac{s}{d}}.
\end{equation}
\end{definition}
According to \cite{Brandolini:2014oz}, low-cardinality cubature sequences with positive weights are
qMC systems:
\begin{proposition}[\cite{Brandolini:2014oz}]\label{th:brandolini}
  For $1 \le p \le \infty$ and $s>d/p$, any low-cardinality cubature sequence
  $\{(x^i_j,\omega^i_j)\}_{j=1}^{n_i}$ with positive weights is a qMC system
  for $W^s_p(\mathcal{M})$.
\end{proposition}

For $1\leq p\leq p'\leq\infty$, due to \eqref{eq:def norm sob}, the space
$W^s_{p'}(\mathcal{M})$ is continuously embedded into $W^s_{p}(\mathcal{M})$
with $\|f\|_{W^s_{p'}}\leq \|f\|_{W^s_{p}}$, for $f\in W^s_{p'}(\mathcal{M})$.
Thus, if $s>d/p$ and $\{(x^i_j,\omega^i_j)\}_{j=1}^{n_i}$ is a qMC system for
$W^s_p(\mathcal{M})$, then it is also a sequence of qMC systems for
$W^s_{p'}(\mathcal{M})$. If $\mathcal{M}$ is the sphere, then the latter becomes
\cite[Theorem 4.2]{Brauchart:2015ec}. Therefore, $p=1$ is the strongest
requirement among the qMC properties.

\begin{remark}
  If $\{(x^i_j,\omega^i_j)\}_{j=1}^{n_i}$ is a qMC system for some $p\geq 1$,
  then Proposition \ref{prop:11} yields that its covering radii are bounded by
 \begin{equation*}
 \rho_i\lesssim  n_i^{-\frac{1}{d}(\frac{s}{s+d(1-1/p)})}.
 \end{equation*}
 Thus, qMC systems for $p=1$ provide
\begin{equation*}
\rho_i\lesssim \big[\wce(\{(x^i_j,\omega^i_j)\}_{j=1}^{n_i},W^s_1(\mathcal{M}))\big]^{1/s}\lesssim n_i^{-\frac{1}{d}},
\end{equation*}
so that Theorem \ref{th:basis} is a consequence of Propositions \ref{prop:11} and \ref{th:brandolini}. Hence, to complete the proof of Theorem \ref{th:basis}, it only remains to verify Proposition \ref{prop:11}, which is the topic of the subsequent section.
\end{remark}

Before we proceed to the proof of Proposition \ref{prop:11}, we shall discuss a
method to compute the worst case error of integration in $W^s_2(\mathcal{M})$,
the latter being a Hilbert space with inner product
\begin{equation}\label{eq:inner product on W}
\langle f,g\rangle_{W^s_2} = \sum_{\ell=0}^\infty (1+\lambda_\ell)^{s} \hat{f}(\ell)\overline{\hat{g}(\ell)},\quad f,g\in W^s_2(\mathcal{M}).
\end{equation}
The Bessel kernel $K^s_B:\mathcal{M}\times \mathcal{M}\rightarrow \R$ given by 
\begin{equation}\label{eq:besel}
K^s_B(x,y)=\sum_{\ell=0}^\infty (1+\lambda_\ell)^{-s}\varphi_\ell(x)\overline{\varphi_\ell(y)}
\end{equation}
is the reproducing kernel for $W^s_2(\mathcal{M})$ with respect to the inner
product \eqref{eq:inner product on W} provided that $s>d/2$. For later reference
we consider a slightly more abstract setting.
If $K:\mathcal{M}\times \mathcal{M}\rightarrow\R$ is a reproducing kernel for
some reproducing kernel Hilbert space $H_K$ of continuous functions on
$\mathcal{M}$, then the worst case error of integration is
\begin{align}\label{eq:wce und so}
\wce(\{(x_j,\omega_j)\}_{j=1}^{n},H_K)^2 &= \sum_{j,j'=1}^n \omega_j\omega_{j'}K(x_j,x_{j'})-2\sum_{j=1}^n \omega_j \int_{\mathcal{M}}K(x_j,x)d\mu(x) \\
&\qquad +\int_{\mathcal{M}}\int_{\mathcal{M}}K(x,y)d\mu(x)d\mu(y),\nonumber
\end{align}
cf.~\cite[Theorem 2.7]{Graf:2013zl}, see also \cite{Nowak:2010rr}. Note, the
norm in $H_{K}$ is uniquely determined by its reproducing kernel $K$. If $K$ has
the Fourier expansion
\begin{equation}\label{eq:kernel fourier expansion}
K(x,y)=\sum_{\ell=0}^\infty \hat{K}(\ell) \varphi_\ell(x)\overline{\varphi_\ell(y)}
\end{equation}
with $\hat{K}(\ell)\in\C$, for $\ell=0,1,\ldots$, 
then \eqref{eq:wce und so} becomes 
\begin{equation}\label{eq:KsB cub formel}
\wce(\{(x_j,\omega_j)\}_{j=1}^{n},H_K)^2 = \sum_{j,j'=1}^n\omega_j\omega_{j'} K(x_j,x_{j'}) +\hat{K}(0)(1-2\sum_{j=1}^n\omega_j),
\end{equation}
where we assume without loss of generality $\varphi_0\equiv 1$. For the Bessel kernel $K^s_B$, we observe $\hat{K}^s_B(\ell)=(1+\lambda_\ell)^{-s}$ with $\hat{K}^s_B(0)=1$

We conclude this section by the following result on the worst case error of
uniformly distributed random points, for which constant weights $\omega_j=1/n$
are the natural choice.
\begin{proposition}\label{th:random inde}
If $K$ is a reproducing kernel on $\mathcal{M}$ and $x_1,\ldots,x_n$ are random points on $\mathcal{M}$, independently identically distributed according to $\mu$, then it holds
\begin{equation*}
  \sqrt{\mathbb{E} \Big[ \wce(\{(x_j,\frac{1}{n})\}_{j=1}^{n},H_K)^2\Big]}  = c n^{-\frac{1}{2}},
\end{equation*}
where 
\begin{equation}\label{eq:const c}
c^2=\int_{\mathcal{M}}K(x,x)d\mu(x) - \int_{\mathcal{M}}\int_{\mathcal{M}}K(x,y)d\mu(x)d\mu(y).
\end{equation}
\end{proposition}
By applying \eqref{eq:kernel fourier expansion}, the constant \eqref{eq:const c} is $c^2=\sum_{\ell=1}^\infty \hat{K}(\ell)$. For the Bessel kernel $K^s_B$, the condition $s>d/2$ implies that
$\frac{s}{d}>\frac{1}{2}$, so that on average qMC systems indeed perform
better than random points for smooth functions. The proof of Proposition \ref{th:random inde} follows from the lines
in \cite{Brauchart:fk} when replacing the sphere by $\mathcal{M}$. In fact, Proposition \ref{th:random inde} is already contained in \cite[Corollary 2.8]{Graf:2013zl}, see also \cite{Nowak:2010rr}.

\section{Proof of Proposition \ref{prop:11}}\label{sec:4}
The proof of Proposition \ref{prop:11} relies on findings in
\cite{Brandolini:2014oz}. We recapitulate the following localization result,
which is one of the key ingredients for the proof of Proposition~\ref{prop:11}.
\begin{lemma}\label{lemma:brando}
For $s\in\R$ and $h:\R\rightarrow \R$ being a smooth function, supported in the interval $[\frac{1}{2},2]$, the kernel
\begin{equation*}
K^s_r(x,y):= \sum_{\ell=0}^\infty h(\frac{\lambda_\ell}{r}) (1+\lambda_\ell)^{-s/2} \varphi_\ell(x)\overline{\varphi_\ell(y)}
\end{equation*}
is bounded by 
\begin{equation}\label{eq:lem bran}
\big\| \int_\mathcal{M} |K^{s}_{r}(\cdot,y)|d\mu(y)\big\|_{L_\infty} \lesssim r^{-s},\qquad r>0,
\end{equation}
where the constant does not depend on $r$.
\end{lemma}
\begin{proof}
According to \cite[Lemma 2.8]{Brandolini:2014oz}, the estimate 
\begin{equation*}
|K^s_r(x,y)|\lesssim r^{d-s} (1+r \dist(x,y))^{-(d+1)},\quad r>0,
\end{equation*}
holds and leads to the requested assertion
\begin{align*}
\sup_{x\in\mathcal{M}} \big| \int_\mathcal{M} |K^{s}_{r}(x,y)|d\mu(y)\big| & \lesssim r^{d-s} \sup_{x\in\mathcal{M}} \int_\mathcal{M} (1+r \dist(x,y))^{-(d+1)}d\mu(y)\\
&\lesssim r^{-s}\sup_{x\in\mathcal{M}} \int_\mathcal{M} r^d(1+r \dist(x,y))^{-(d+1)}d\mu(y)\\
&\lesssim r^{-s} . \qedhere
\end{align*}
\end{proof}
We shall make use of Lemma \ref{lemma:brando} to verify the following result.
\begin{lemma}\label{lemma:key 2}
Let $R, s >0$, and $1\leq p\leq \infty$ be fixed. For any $0<r\leq R$ and $z \in \mathcal M$, there is a function
$f := f_{r,z} \in\mathcal{C}^\infty(\mathcal{M})$ with support in $B_{r}(z)$, such that
\begin{equation}\label{eq:lemma key 2}
 \|f\|_{W^s_p} \lesssim r^{-s+d/p}, \qquad r^{d}\lesssim \int_{\mathcal{M}}f(x)d\mu(x),
\end{equation}
where the constants do not depend on $z$ or $r$. 
\end{lemma}
\begin{proof}
As in \cite[Proof of Theorem 2.16]{Brandolini:2014oz}, for any radius $0<r\leq R$ and $z \in \mathcal M$, there is a function
$f := f_{r,z} \in\mathcal{C}^\infty(\mathcal{M})$ with support in $B_{r}(z)$, such that, for $\ell=0,1,2,\ldots$,
\begin{equation}\label{eq:basis for estimates}
  \|\Delta^\ell f\|_{L_\infty} \lesssim r^{-2\ell},\qquad r^{d}\lesssim \int_{\mathcal{M}}f(x)d\mu(x),
\end{equation}
where by compactness of $\mathcal M$ the constants do not depend on $z$ or
$r$.   

Similarly as in \cite[Proof of Theorem 2.16]{Brandolini:2014oz}, we bound the
norm $\|f\|_{W_{p}^{s}} = \|(I+\Delta)^{s} f \|_{L_{p}}$ by using a dyadic
decomposition of unity of the Fourier domain, see, for instance,
\cite{Ullrich:2014oj}. That is, we set $h(x):=g(x)-g(2x)$ for some smooth
function $g:\R\to \R$ satisfying
\[
  g(x)=\begin{cases} 1,& x\leq 1,\\ 0,& x \ge 2, \end{cases}
\]
and obtain $\supp(h)\subset [1/2,2]$ with
\begin{equation*}
  \sum_{m=-\infty}^\infty h(2^{-m}x) = \begin{cases}
    1,&x > 0,\\ 0,& \text{else}.
\end{cases}
\end{equation*}
Using the Fourier expansion $f=\sum_{\ell=0}^\infty \hat{f}(\ell)\varphi_\ell$
we arrive at
\begin{equation*}
(I+\Delta)^{s/2}f  = \hat{f}(0) +  \sum_{\ell=0}^\infty \;\sum_{m=-\infty}^\infty h(2^{-m}\lambda_\ell)(1+\lambda_\ell)^{s/2} \hat{f}(\ell)\varphi_\ell.
\end{equation*}
Fixing an integer $L>s/2$ and applying the notation of Lemma \ref{lemma:brando} with $\hat{f}(\ell)=\int_\mathcal{M} f(y)\overline{\varphi_\ell(y)}d\mu(y)$ lead to
\begin{align*}
(I+\Delta)^{s/2}f &=  
 \hat{f}(0) +   \sum_{2^m\leq 1/r} \int_\mathcal{M} K^{-s}_{2^m}(\cdot,y)f(y)d\mu(y).\\
&\qquad\quad + \sum_{2^m>1/r}\int_\mathcal{M} K^{2L-s}_{2^m}(\cdot,y)(I+\Delta)^{L} f(y)d\mu(y).
\end{align*}
We shall now bound the three terms of the right hand side separately. First, the
H\"older inequality with \eqref{eq:basis for estimates} for $\ell=0$ yields
\begin{equation*}
|\hat{f}(0)|=\big|\int_\mathcal{M}f(y)d\mu(y) \big|\leq \|f\|_{L_\infty} \mu(\mathcal{M})\lesssim 1\lesssim r^{-s},
\end{equation*}
where the last estimate is due to $r$ being bounded from above and $s>0$. 
To bound the second term, we apply Lemma \ref{lemma:brando} and derive
\begin{align*}
\|\sum_{2^m\leq 1/r} \int_\mathcal{M} K^{-s}_{2^m}(\cdot,y)f(y)d\mu(y)\|_{L_\infty} 
&\leq \sum_{2^m\leq 1/r}  \| \int_\mathcal{M} |K^{-s}_{2^m}(\cdot,y)|d\mu(y)\|_{L_\infty}   \|  f\|_{L_\infty}  \\
&\lesssim \sum_{2^m\leq 1/r} \| \int_\mathcal{M} |K^{-s}_{2^m}(\cdot,y)|d\mu(y)\|_{L_\infty}\\
&\lesssim \sum_{2^m\leq 1/r} 2^{ms}\lesssim r^{-s}.
\end{align*}
Lemma \ref{lemma:brando} also leads to a bound on the third term by 
\begin{align*}
\|\sum_{2^m > 1/r} \int_\mathcal{M} K^{2L-s}_{2^m}(\cdot,y)(I+\Delta)^{L}f(y)d\mu(y)\|_{L_\infty} 
& \\
 & \hspace{-3cm} \lesssim  \sum_{2^m > 1/r}\| \int_\mathcal{M} K^{2L-s}_{2^m}(\cdot,y)(I+\Delta)^{L}f(y)d\mu(y)\|_{L_\infty} \\
&  \hspace{-3cm} \lesssim \sum_{2^m > 1/r}\| \int_\mathcal{M} K^{2L-s}_{2^m}(\cdot,y)d\mu(y)\|_{L_\infty} \| (I+\Delta)^{L} f\|_{L_\infty} \\
& \hspace{-3cm} \lesssim \sum_{2^m > 1/r} 2^{-(2L-s)m}   r^{-2L}\\
& \hspace{-3cm} \lesssim r^{2L-s} r^{-2L}=r^{-s},
\end{align*}
where we have also applied \eqref{eq:basis for estimates} and $r$ being bounded from above. 

Thus, we obtain 
\begin{equation}\label{eq:aa}
\|(I+\Delta)^{s/2} f \|_{L_\infty} \lesssim r^{-s},
\end{equation}
where the constants do not depend on $z$ or $r$ To cover the range
$1\leq p<\infty$, we recall that $f$ is supported on $B_{r}(z)$, so that the
H\"older inequality with \eqref{eq:inequality balls} yields
\begin{equation}\label{eq:bb}
  \|(I+\Delta)^{s/2} f \|^p_{L_p} \lesssim r^{-ps}\mu(B_{r}(z))\lesssim r^{-sp+d},
\end{equation}
which concludes the proof. 
\end{proof}

We are now prepared to complete the proof of our main theoretical result.
\begin{proof}[Proof of Proposition \ref{prop:11}]
For a given point set $\{ x_{j} \}_{j=1}^{n} \subset \mathcal M$ with covering radius $\rho$, let $z$ be
a center of a maximal hole, i.e.,
\begin{equation}\label{eq:ringi}
  \mathring{B}_{\rho}(z) \cap \{x_j\}_{j=1}^{n}=\emptyset,
\end{equation}
where $\mathring{B}_{\rho}(z)=\{x\in\mathcal{M} : \dist(x,z)<\rho\}$ denotes the interior of $B_\rho(z)$. 
Note that $\rho$ is bounded by the diameter of $\mathcal{M}$. 
Let $f=f_{\rho,z} \in \mathcal{C}^{\infty}(\mathcal M)$ be as in Lemma \ref{lemma:key 2}, i.e., $\supp(f) \subset B_{\rho}(z)$ and \eqref{eq:lemma key 2} holds with $r=\rho$. Since $f$ must vanish outside of $\mathring{B}_{\rho}(z)$, \eqref{eq:ringi} implies $f(x_j)=0$, for all $j=1,\ldots,n$. Thus, the definition of the worst case
error of integration yields
\begin{align*}
\wce(\{(x_j,\omega_j)\}_{j=1}^{n},W^s_p(\mathcal{M})) &\geq \frac{\big|\int_{\mathcal{M}}f(x)d\mu(x) - 0\big|}{\|f\|_{W^s_p}}\gtrsim \frac{\rho^{d}}{\rho^{-s+d/p}}=\rho^{s+d/q},
\end{align*}
with $1/p+1/q=1$, and the constant does not depend on $z$ or $\rho$. 
\end{proof}

\section{Numerical experiments for the Grassmannian manifold}\label{sec:nums}
This section is dedicated to illustrate the results of the previous sections for
the special case of the Grassmannian manifold, i.e., the collection of
$k$-dimensional linear subspaces in $\R^m$, which we identify with the set of
orthogonal projectors on $\R^m$ of rank $k$, denoted by
\begin{equation}
\label{eq:embGkm}
\G_{k,m} := \{ P \in \R^{m\times m}: P^\top=P,\;P^{2}=P ,\; \tr(P)=k \}.
\end{equation}
Hence, the Grassmannian $\mathcal G_{k,m}$ can be considered as a
$k(m-k)$-dimensional submanifold of the Euclidean space
$\R^{m^{2}} \cong \R^{m\times m}$. Moreover the Euclidean space $\R^{m^{2}}$
induces a Riemannian metric, which in turn yields the canonical probability
measure and the canonical geodesic distance on the Grassmannian
$\mathcal G_{k,m}$ denoted by $\mu_{k,m}$ and $\dist_{k,m}$, respectively. In
particular, the geodesic distance $\dist_{k,m}(P,Q)$ between $P,Q\in\G_{k,m}$ is
proportional to the $2$-norm of the corresponding principal angles
$\theta_1,\ldots\theta_k$ between the subspace associated to $P$ and $Q$. More
precisely, it can be computed by
\begin{equation}\label{eq:dd}
  \dist_{k,m}(P,Q) = \sqrt 2\sqrt{\theta^2_1+\ldots+\theta^2_k}, 
\end{equation}
where $\theta_i=\arccos(\sqrt{y_i})$ and $y_1,\ldots,y_k$ are the $k$-largest
eigenvalues of $PQ$ (counted with multiplicities). Note that the factor of $\sqrt 2$
accounts for the particular embedding \eqref{eq:embGkm} since then
\begin{equation}
  \label{eq:dequivFrob}
  \dist_{k,m}(P,Q) = \| P - Q\|_{\mathrm F} + o(\|P - Q\|_{\mathrm F}^{2}),
  \qquad P,Q \in \mathcal G_{k,m},
\end{equation}
where $\| X \|_{\mathrm F}$ is the Frobenius-norm of $X \in \R^{m\times m}$.

\subsection{Cubature points}\label{subsec:cub}
Theorem \ref{th:basis} tells us that low-cardinality cubature points with positive weights inherit
asymptotically optimal covering radii. To illustrate this result, we first
construct a sequence of cubature points. It is known that any collection of
points $\{P_j\}_{j=1}^n\subset\mathcal{G}_{k,m}$ satisfies
\begin{equation}\label{cub num est}
  \frac{1}{n^{2}} \sum_{j,j'=1}^{n} \tr(P_{j},P_{j'})^{i} \ge
  \int_{\G_{k,m}}\int_{\G_{k,m}} \tr(P,Q)^{i}  d \mu_{k,m}(P)   d \mu_{k,m}(Q),
\end{equation}
for $i=1,2,\ldots$, cf.~\cite{Bachoc:2010aa}. According to \cite{Breger:2016vn}, equality in \eqref{cub num est} yields a design of strength $t_i=i\frac{2}{\sqrt{k}}$, see also \cite{Ehler:2014zl}. Hence, \eqref{cub num est} provides us with a simple approach to numerically  compute cubature points by minimization and checking for equality. 

Our numerical experiments shall focus on $\mathcal{G}_{2,4}$, which has
dimension $d=4$, so that for low-cardinality cubature sequences the number of
cubature points must satisfy $n_{i} \asymp i^{4} \asymp {t_i}^{4}$. Indeed, we
have chosen
\[
  n_{i} =  \Big\lfloor  \tfrac13 (i+1)^2(1+i+\tfrac12 i^2) \Big\rfloor
\]
and computed points $\{P_{j}^{i}\}_{j=1}^{n_{i}}\subset\G_{2,4}$, for $i=1,\dots,14$, by a nonlinear conjugate gradient method on manifolds, cf.~\cite{Absil:2008qr,Breger:2016vn,Graf:2013zl}, such that 
\begin{equation*}
\big| \int_{\mathcal{G}_{2,4}} f(P)d\mu_{2,4}(P) - \frac{1}{n_i} \sum_{j=1}^{n_i} f(P_j) \big|  < 10^{-7},
\end{equation*}
for all $f\in \Pi_{t_{i}}$ with $\|f\|_{L_2}\leq 1$. Although the worst case
error of integration in $\Pi_{t_{i}}$ may not be zero exactly, we shall refer to
$\{P_{j}^{i}\}_{j=1}^{n_{i}}$ in the following simply as $t_{i}$-designs. 

\subsection{Integration}

In view of bounding the covering radius, we may want to provide numerical experiments on the worst case error of integration in Sobolev spaces $W^s_p(\mathcal{G}_{k,m})$ for $p=1$. However, determining the worst case error is a tough task in general. For $p=2$ and $s>k(m-k)/2$, on the other hand, we are dealing with reproducing kernel Hilbert spaces, in which we can invoke \eqref{eq:KsB cub formel} provided that its reproducing kernel is numerically accessible.  

Our first numerical experiments are about integration in $W^{7/2}_2(\G_{2,4})$, so that the worst case error is indeed given by \eqref{eq:KsB cub formel}. However, the infinite series of the Bessel kernel $K^s_B$ in \eqref{eq:besel} is numerically cumbersome, so that we consider the
positive definite kernel
\begin{equation*}    
  K_{1}(P,Q)  = k_{1}(\tr(PQ)), \qquad P,Q\in\mathcal{G}_{2,4},
\end{equation*}
where
\[
  k_{1}(r) = \sqrt{(2-r)^{3}} + 2 r, \qquad 0 \le r \le 2.
\]
Due to a comparison to the Bessel kernel, cf.~\cite{Brandolini:2014oz,Breger:2016vn}, its reproducing kernel Hilbert space $H_{K_{1}}$ is the Sobolev space $W^{7/2}_2(\G_{2,4})$ equipped with an equivalent norm, i.e., the two norms $\|\cdot\|_{W^{7/2}_2}$ and $\|\cdot\|_{K_1}$ are comparable. This implies
\begin{align*}
\wce(\{(P_{j},\tfrac1n)\}_{j=1}^{n},W^{7/2}_2(\G_{2,4})) & \asymp \wce(\{(P_{j},\tfrac1n)\}_{j=1}^{n},H_{K_{1}})
\end{align*}
where the constants are independent of the point sets.
According to Proposition \ref{th:brandolini}, the error
$\wce(\{(P_{j}^{i},\tfrac{1}{n_{i}})\}_{j=1}^{n_{i}},H_{K_1}
)$ decays as $n_{i}^{-7/8}$, for low-cardinality cubature points with positive
weights. Hence, we expect a linear behavior with slope $-7/8$ in logarithmic
error plots.

In a second numerical experiment on integration, we shall consider a second positive definite kernel $K_2$ given by 
\begin{equation*}    
K_{2}(P,Q)  = k_{2}(\trace(PQ)), \qquad P,Q\in\mathcal{G}_{2,4},
\end{equation*}
where
\[
  k_{2}(r) = \tfrac{3}{2}\exp(-(2-r)), \qquad 0 \le r \le 2.
\]
Its reproducing kernel Hilbert space $H_{K_{2}}$ satisfies
\begin{equation}\label{eq:HK2}
H_{K_{2}} \subset \bigcap_{s>2}W^{s}_2(\G_{2,4}).
\end{equation}
According to Theorem \ref{th:brandolini}, a low-cardinality cubature sequence
with positive weights is a qMC system for any $s>2$. Due to \eqref{eq:HK2}, we
expect a super linear behavior of $\wce(\{(P_{j}^{i},\tfrac{1}{n_{i}})\}_{j=1}^{n_{i}},H_{K_2})$
in logarithmic plots.

We now further specify the worst case error in $H_{K_{1}}$ and $H_{K_{2}}$ via \eqref{eq:KsB cub formel}.
\begin{lemma}\label{lemma:compt K wce}
  The worst case errors in $H_{K_{1}}$ and $H_{K_{2}}$ are 
    \begin{align}
      \wce(\{(P_{j},\tfrac1n)\}_{j=1}^{n},H_{K_{1}} )^{2} &= \frac{1}{n^{2}}
      \sum_{j,j'=1}^{n} K_{1}(P_{j},P_{j'}) \label{eq:aans}\\ &\qquad \qquad - \big(2 + \frac{74}{75} \sqrt 2 -
      \frac{2}{5} \log(1+\sqrt 2)\big),\nonumber\\
      \wce(\{(P_{j},\tfrac1n)\}_{j=1}^{n},H_{K_2})^{2} &= \frac{1}{n^{2}}
      \sum_{j,j'=1}^{n} K_{2}(P_{j},P_{j'}) - \tfrac32 \exp(-1) \mathrm{Shi}(1),\label{eq:zwoaaaa}
    \end{align}
  respectively, where $\mathrm{Shi}(x) = \int_{0}^{x} \frac{\sinh(t)}{t} d t$ is
  the hyperbolic sine integral.
\end{lemma}

\begin{proof}
  In view of \eqref{eq:KsB cub formel} it remains to compute the $0$-th Fourier
  coefficients
  \[
    \hat{K}_{i}(0) = \int_{G_{2,4}} K_{i}(P,Q) d\mu_{2,4}(P)
    d\mu_{2,4}(Q), \qquad i=1,2.
  \]
According to \cite[Example 4.3]{Davis:1999dn}, the orthogonal invariance of $K_1$ and $K_2$ with the variable transformation $\xi_\pm=\cos(\theta_1\pm \theta_2)$, where $\theta_1,\theta_2$ are the principal angles between $P$ and $Q$, 
yield
  \[
    \int_{G_{2,4}} K_{i}(P,Q) d\mu_{2,4}(P) d\mu_{2,4}(Q) = \int_{-1}^{1}
    \int_{|\xi_{-}|}^{1} k_{i}(1+\xi_{-}\xi_{+}) d\xi_{+}d\xi_{-}, \qquad i=1,2.
  \]
  The symmetry of the function $(\xi_{-},\xi_{+}) \mapsto \xi_{-}\xi_{+}$
  leads to
  \[
    \hat{K}_{i}(0) = \frac14 \int_{-1}^{1}\int_{-1}^{1}
    k_{i}(1+\xi_{-}\xi_{+}) d\xi_{+}d\xi_{-}, \qquad i=1,2.
  \]
For $i=1$, we arrive at 
    \begin{align*}
      \hat{K}_{1}(0) & = \frac14 \int_{-1}^{1}\int_{-1}^{1}
      \big((1-\xi_{-}\xi_{+})^{\frac32} + 2 \xi_{-}\xi_{+} + 2\big)
      d\xi_{+}d\xi_{-} \\
&      = 2 + \int_{-1}^{1} \tfrac{(1+\xi_{-})^{\frac52}-(1-\xi_{-})^{\frac52}}{10
        \xi_{-}} d\xi_{-}.
\end{align*}
The assertion \eqref{eq:aans} is then checked by a computer algebra system. 

For $i=2$, we obtain
\begin{align*}
      \hat{K}_{2}(0) & = \frac38 \int_{-1}^{1}\int_{-1}^{1}
      \exp(\xi_{-}\xi_{+}-1) d\xi_{+}d\xi_{-} \\
&      = \frac38 \exp(-1)\int_{-1}^{1}
      \tfrac{\exp(\xi_{-})-\exp(-\xi_{-})}{\xi_{-}} d\xi_{-},
    \end{align*}
so that \eqref{eq:zwoaaaa} follows immediately. 
\end{proof}

Figure~\ref{fig:1} shows logarithmic plots of
\begin{equation*}
\wce(\{(P^i_{j},\tfrac{1}{n_{i}})\}_{j=1}^{n_{i}},H_{K_1}
),\qquad \wce(\{(P_{j},\tfrac{1}{n_{i}})\}_{j=1}^{n_{i}},H_{K_2})
\end{equation*}
for the cubature points $\{P_{j}^{i}\}_{j=1}^{n_{i}}\subset\G_{2,4}$,
$i=1,\dots,14$, from Section \ref{subsec:cub}. For comparison, it also depicts
the worst case error of integration for random points, independently identically
distributed according to $\mu_{2,4}$, cf.~Proposition \ref{th:random inde}.
Indeed, we observe the superior integration quality of the $t_i$-designs over
random points. The theoretical results in Propositions \ref{th:brandolini} and
\ref{th:random inde} are in perfect accordance with the numerical experiment.
The integration errors of the random points scatter around the expected
integration error with rate $n_i^{-1/2}$ in both cases, $H_{K_1}$ and $H_{K_2}$.
Cubature points achieve the optimal rate of $n^{-7/8}_i$ for functions in the
Sobolev space $W^{7/2}_2(\G_{2,4})$. Due to
$H_{K_{2}}\subset \bigcap_{s>2}W^{s}_2(\G_{2,4})$, we observe the super linear
behavior in the logarithmic plots for
$\wce(\{(P_{j}^{i},\tfrac{1}{n_{i}})\}_{j=1}^{n_{i}},H_{K_2})$.


\begin{figure}[t]
\begin{center}
 \includegraphics[width=.8\textwidth]{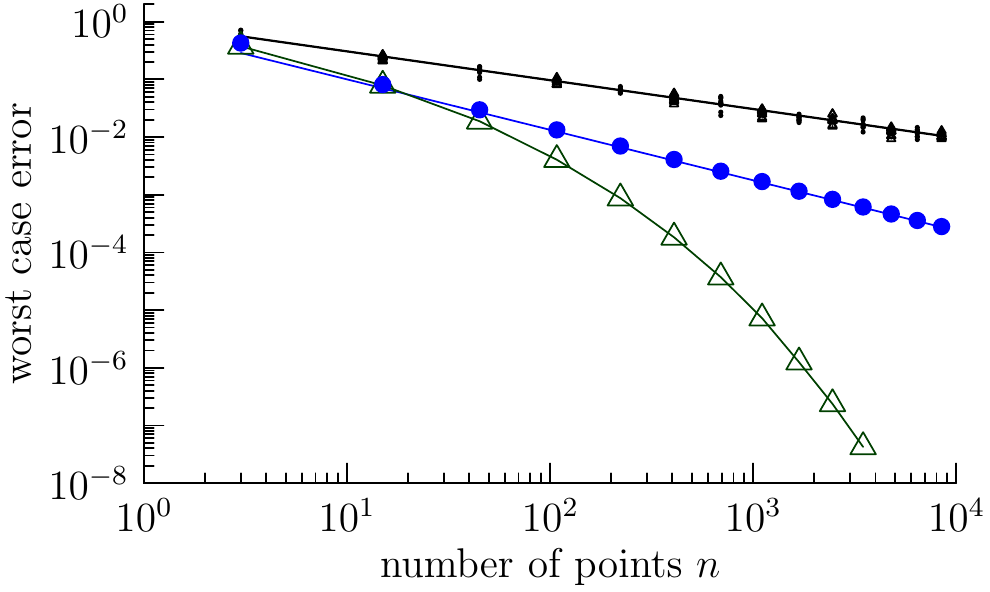}
\end{center}
\caption{Integration by random sampling vs.~$t$-designs. Dots refer to the worst
  case error in $H_{K_1}$, triangles to $H_{K_2}$. Big symbols correspond to
  $t_i$-designs, small symbols to $n_{i}$ random points. The black line is the
  expectation of the worst case error and has slope
  $-1/2$,~cf.~Proposition~\ref{th:random inde}, and the blue line's slope is
  $-7/8$.}\label{fig:1}
\end{figure}

To conclude this section on numerical integration, we point out that Lemma \ref{lemma:compt K wce} provides analytic expressions for the worst case errors in particular reproducing kernel Hilbert spaces. If we go beyond Hilbert spaces, then numerically computing worst case errors in Sobolev spaces becomes very challenging as is illustrated by the following remark. 

 \begin{remark}
   Let $I_2$ denote the $4$ by $4$ matrix with two ones in the left upper
   diagonal entries and zeros elsewhere. We shall consider the integration error with respect to a sequence of low-cardinality cubatures $\{(P^i_{j},\omega^i_j)\}_{j=1}^{n_{i}}$ for $\Pi_{t_i}$ in $\mathcal{G}_{2,4}$. Without loss of generality we assume $P^i_1=I_2$, for $i=1,2,\ldots$. 
   The function $f_1(P):=K_1(I_2,P)$, for
   $P\in\mathcal{G}_{2,4}$, is contained in
   $W^{7/2}_2(\mathcal{G}_{2,4})$, so that its integration error decays at least
   as fast as $n^{-7/8}$. Numerical experiments suggest that it decays exactly at this rate, so that $f_1$ seems to be a single representative for the worst case error of integration in the reproducing kernel Hilbert space 
   $W^{7/2}_2(\mathcal{G}_{2,4})$. 
   
   Now, we point out, although $f_1\in W^3_\infty(\mathcal{G}_{2,4})$ with
   $f_1\not\in W^{3+\epsilon}_\infty(\mathcal{G}_{2,4})$, for all $\epsilon>0$,
   its integration error decays faster than the worst case error in
   $W^3_\infty(\mathcal{G}_{2,4})$, which is just $n^{-3/4}$.
    \end{remark}


\subsection{Approximating covering radii by random points}\label{sec:app}
Given a sequence of point sets
$\{P_{j}^{i}\}_{j=1}^{n_{i}}\subset\mathcal{G}_{k,m}$, for $i=1,2,\ldots$, it is
a tough task to numerically determine the exact covering radii $\rho_i$. In
order to obtain a reasonable approximation, we generate $n\gg n_i$ random points
$\{R_{j'}\}_{j'=1}^n\subset\mathcal{G}_{k,m}$ and determine
\begin{equation}\label{eq:approxi by rand}
  \hat{\rho}_{i,n}:=\max_{1\leq j'\leq n}  \min_{1\leq j\leq n_i} \dist_{k,m}(R_{j'},P^i_j). 
\end{equation}
Note that $\rho_i$ is sandwiched by 
\begin{equation}\label{eq:estimate upper lower}
\hat{\rho}_{i,n} \leq \rho_i\leq \hat{\rho}_{i,n}+\varrho_n,
\end{equation}
where $\varrho_n$ is the covering radius of the random points $\{R_{j'}\}_{j'=1}^n$. Supposed that $\varrho_n$ is sufficiently small, $\hat{\rho}_{i,n}$ is a decent approximation of $\rho_i$. 

We aim to provide evidence that the computed approximation $\hat{\rho}_{i,n}$ is
sufficiently accurate to numerically illustrate Theorem~\ref{th:basis}. To
ballpark $\varrho_n$, recall that the covering radius $\varrho_n$ of $n$ random
points, independently distributed according to $\mu_{k,m}$ on
$\mathcal{G}_{k,m}$, behaves asymptotically as stated in \eqref{eq:random
  asympt} with $d=k(m-k)$. More precisely, \cite[Corollary 3.3]{Reznikov:2015zr}
yields together with the relation \eqref{eq:dequivFrob} that
\begin{equation}\label{eq:asym}
\lim_{n\rightarrow \infty}\mathbb{E} \varrho_n \left(\frac{n}{\log(n)}\right)^{\frac{1}{k(m-k)}} = \left(\frac{\vol(\mathcal{G}_{k,m})}{\vol(\mathbb{B}_{k(m-k)})}\right)^{\frac{1}{k(m-k)}},
\end{equation}
where $\mathbb{B}_d$ is the unit ball in $\R^d$ and $\vol$ are the canonical
volumes induced by the Euclidean metric. Thus, for large $n$, we expect that $\varrho_n$
behaves as
\begin{equation}\label{eq:express for}
  \varrho_{n} \approx \left(\frac{\vol(\mathcal{G}_{k,m})}{\vol(\mathbb{B}_{k(m-k)})}\frac{n}{\log(n)}\right)^{\frac{1}{k(m-k)}}.
\end{equation}
The volume of the Grassmannian is
\begin{equation}
\label{eq:volGkm}
\vol(\mathcal{G}_{k,m})=\frac{\Gamma_k(k/2)}{\Gamma_k(m/2)}(2\pi)^{k(m-k)/2}
\end{equation}
where $\Gamma_k$ is the multivariate Gamma function
\begin{equation*}
\Gamma_k(x)=\pi^{k(k-1)/4}\prod_{i=1}^k \Gamma(x+(1-i)/2), \qquad x > \frac12(k-1).
\end{equation*}
Note, the formula \eqref{eq:volGkm} differs by a factor of $\sqrt 2^{k(m-k)}$ from
\cite[Eq.~(1.4.11)]{Chikuse:2003aa} due to our additional scaling of $\sqrt 2$ in the
geodesic distance,~cf.~\eqref{eq:dd}. 
The volume of $\mathbb{B}_{d}$ in $\R^d$ is 
\begin{equation*}
\vol(\mathbb{B}_{d}) = \frac{\pi^{d/2}}{\Gamma(d/2+1)}.
\end{equation*}
Hence, for $k=2$, $m=4$, the expression \eqref{eq:express for} reduces to
\begin{equation}
\label{eq:expest}
  \varrho_{n} \approx 2n^{-\frac{1}{4}} \log(n)^{\frac{1}{4}}.
\end{equation}
For $n=10^{7}$, we obtain $\varrho_{n}\approx 0.0713$, which is significantly
smaller than any of the computed $\hat{\rho}_{n,i}$. According to these
considerations, we argue that our numerical computations of the covering radius
are sufficiently reliable in view of \eqref{eq:estimate upper lower}.

Figure~\ref{fig:2} shows logarithmic plots of the estimated covering radii for
the cubature points $\{P_{j}^{i}\}_{j=1}^{n_{i}}$, $i=1,\dots,14$. For
comparison it also depicts estimated covering radii for random points. We
observe the desired relationship $\rho_{i}\asymp n_i^{-1/4}$,
cf.~Theorem~\ref{th:basis}, and the estimate \eqref{eq:expest}
of the expected covering radius for random points becomes more accurate for large $n$. 

\begin{figure}
\centering
\includegraphics[width=.8\textwidth]{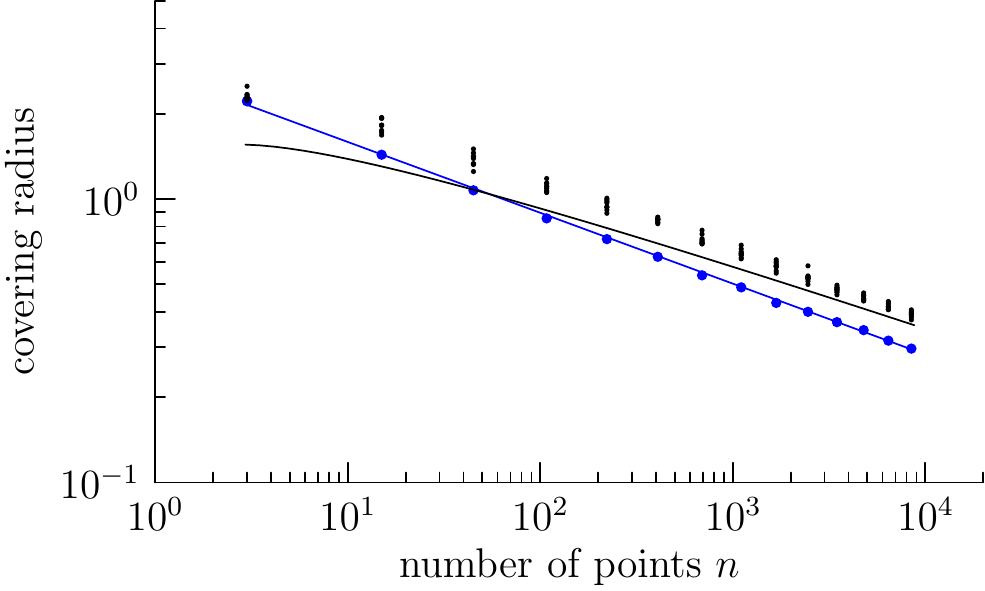}
\caption{Covering radii of random sampling vs. $t$-designs. Big dots correspond
  to $t_{i}$-designs, small dots to $n_{i}$ random points. The blue line's slope is $-1/4$. The black curve
  corresponds to $2n^{-\frac{1}{4}} \log(n)^{\frac{1}{4}}$, which relates to the covering radii $\varrho_n$ via \eqref{eq:expest}.
  }
\label{fig:2}
\end{figure}

\section*{Acknowledgements}
The authors have been funded by the Vienna Science and Technology Fund (WWTF) through project VRG12-009.

\bibliographystyle{amsplain}
\bibliography{../biblio_ehler2}

\end{document}